\documentclass[10pt]{amsart}
\usepackage{amssymb}
\usepackage{epsfig}
\usepackage{url}
\usepackage{setspace}
\usepackage{pdflscape}
\theoremstyle{plain}

\newtheorem{thm}{Theorem}[section]

\newtheorem{rem}[thm]{Remark}
\newtheorem{ques}[thm]{Question}
\newtheorem{conj}[thm]{Conjecture}

\def\bbb{\mathbb}

\renewcommand{\phi}{\varphi}

\newcommand{\N}{\bbb{N}}
\newcommand{\Z}{\bbb{Z}}

\makeatletter
\let\@@pmod\pmod
\DeclareRobustCommand{\pmod}{\@ifstar\@pmods\@@pmod}
\def\@pmods#1{\mkern4mu({\operator@font mod}\mkern 6mu#1)}
\makeatother

\begin{document}

\title[Some observations and speculations on partitions into $d$-th powers]{Some observations and speculations on partitions into $d$-th powers}
\author{Maciej Ulas}

\keywords{partitions into powers, identities, numerical computations}
\subjclass[2010]{05A17, 11P83}
\thanks{The research of the author is supported by the grant of the National Science Centre (NCN), Poland, no. UMO-2019/34/E/ST1/00094}

\begin{abstract}
The aim of this note is to provoke discussion concerning arithmetic properties of function $p_{d}(n)$ counting partitions of an positive integer $n$ into $d$-th powers, where  $d\geq 2$. Besides results concerning the asymptotic behavior of $p_{d}(n)$ a little is known. In the first part of the paper, we prove certain congruences involving functions counting various types of partitions into $d$-th powers. The second part of the paper has experimental nature and contains questions and conjectures concerning arithmetic behavior of the sequence $(p_{d}(n))_{n\in\N}$. They based on our computations of $p_{d}(n)$ for $n\leq 10^5$ in case of $d=2$, and $n\leq 10^{6}$ for $d=3, 4, 5$.
\end{abstract}

\maketitle

\section{Introduction}\label{sec1}

Let $A\subset\N_{+}$ be given and take $n\in\N$. By a partition of a non-negative integer $n$ with parts in $A$, we mean any representation of $n$ in the form
$$
n=a_{1}+\ldots+a_{k},
$$
where $a_{i}\in A$. In particular, if $A=\N_{+}$, then $p(n)$ is the famous partition function introduced by L. Euler and extensively studied by S. Ramanujan. The function $p(n)$ counts the number of partitions with parts in $\N_{+}$, i.e., unrestricted partitions of $n$.

The literature concerning arithmetic properties of functions counting various partitions is enormous. However, the theory is concentrated mainly on the case when the set $A$ is a sum of disjoint arithmetic progressions. In this case, the theory is especially rich because of the connections with modular forms and the general theory of $q$ series (see for example \cite{And1}). One can say that the counting function $A(x)=\#\{a\in A:\;a \leq x\}$ is linear in this case, i.e., $A(x)=O(x)$. There is also a nice theory connected with the set of powers of a fixed integer $m$ (so-called $m$-partitions), i.e., in this case, we have that $A(x)$ has logarithmic growth, i.e., $A(x)=O(\log x)$. On the other side, we know too much about the arithmetic behavior of functions counting partitions into $d$-th powers, where $d\in\N_{\geq 2}$ is fixed. In this case, we have $A=\{k^{d}:\;k\in\N\}$, and thus the growth of $A(x)$ is of the type $x^{1/d}$, and is between two cases mentioned earlier. Let $p_{d}(n)$ denotes the number of partitions of $n$ into $d$-th powers. From the general principles, we know that the ordinary generating function of the sequence $(p_{d}(n))_{n\in\N}$ has the form
$$
P_{d}(q)=\sum_{n=0}^{\infty}p_{d}(n)q^{n}=\prod_{n=1}^{\infty}\frac{1}{1-q^{n^{d}}}.
$$

Very little is know about the arithmetic properties of the sequence $(p_{d}(n))_{n\in\N}$. It seems that up to date the main line of research was the investigations devoted to asymptotic behavior of $p_{d}(n)$. Let us recall that G. H. Hardy and S. Ramanujan claimed \cite{HR}, and E. Maitland Wright proved \cite{MW}, that
$$
\log p_{d}(n)\sim (d+1)\left(\frac{1}{d}\Gamma\left(1+\frac{1}{d}\right)\zeta\left(1+\frac{1}{d}\right)\right)^{\frac{d}{d+1}}n^{\frac{1}{d}}.
$$
The proof of Maitland Wright was very complicated and was simplified by R. Vaughan in the case $d=2$ \cite{V}, and in the general case by A. Gafni \cite{AG}. The proofs of Vaughan and Gafni are based on the circle method approach. Note that recently a new proof, used only saddle point method, was presented by G. Tanenbaum, J. Wu, and Y.-L. Li \cite{TWL}. However, according to our best knowledge, besides identities between partitions into $d$-th powers of various types, which can be deduced from simple manipulations of infinite products and recent result of Ciolan \cite{Cio} who proved that the number of partitions into squares with an even number of parts is asymptotically equal to that of partitions into squares with an odd number of parts, there are no theoretical or experimental results. The absence of such studies was the main motivation for our research.

Let us describe the content of the paper in some details. In Section \ref{sec2} we prove some congruences for functions counting various type of partitions into $d$-th powers, where $d\in\N_{\geq 2}$. In particular, if $A_{2,p_{2}}(n)$ denotes the denotes the number of partitions into $d$-th powers of integers not divisible by $2$ or $p_{2}$, and $B_{2,p_{2}}(n)$ denotes the number of partitions of $n$ into distinct $d$-th powers not divisible by $p_{2}^{d}$, where each part has one among $2^{d}-1$ colors, then $A_{2,p_{2}}(n)\equiv B_{2,p_{2}}(n)\pmod*{2}$.

In Section \ref{sec3}, we present many computational observations based on our computer experiments. In particular, we state several questions and conjectures concerning arithmetic behavior of the sequence $(p_{d}(n))_{n\in\N}$ for $d=2, 3, 4, 5$.

\section{A class of congruences}\label{sec2}

This short section is devoted to the proof of a class of congruences involving partitions into $d$-th powers under certain restrictions. More precisely, let $p_{1}, p_{2}\in\N_{\geq 2}$ and assume that $(p_{1}, p_{2})=1$. Let $A_{p_{1},p_{2}}(n)$ denote the number of partitions into $d$-th powers of integers not divisible by $p_{1}^{d}$ or $p_{2}^{d}$. It is easy to see that the generating function for the sequence $(A_{p_{1},p_{2}}(n))_{n\in\N}$ is the following
$$
\mathcal{A}_{p_{1},p_{2}}(q)=\sum_{n=0}^{\infty}A_{p_{1},p_{2}}(n)q^{n}=\prod_{n=1}^{\infty}\frac{(1-q^{(p_{1}n)^{d}})(1-q^{(p_{2}n)^{d}})}{(1-q^{n^{d}})(1-q^{(p_{1}p_{2}n)^{d}})}.
$$
We prove the following

\begin{thm}\label{thm2}
Let $d\in\N_{+}, n\in\N$ and $p_{1}, p_{2}\in\N_{\geq 2}$. Let $B_{p_{1},p_{2}}(n)$ denote the number of partitions of $n$ into distinct $d$-th powers not divisible by $p_{2}^{d}$, where each part has one among $p_{1}^{d}-1$ colors, and let $C_{p_{1},p_{2}}(n)$ denote the number of partitions of $n$ into $d$-th powers not divisible by $p_{2}^{d}$, where each part has one among $p_{1}^{d}-1$ colors.
\begin{enumerate}
\item If $p_{1}=2$ and $p_{2}$ is odd, then $A_{2,p_{2}}(n)\equiv B_{2,p_{2}}(n)\pmod*{2}$.
\item If $p_{1}\in\mathbb{P}_{\geq 3}$ and $p_{2}\in\N_{\geq 2}, p_{1}\nmid p_{2}$, then for $n\geq 1$ we have
$$
\sum_{i=0}^{n}A_{p_{1},p_{2}}(i)C_{p_{1},p_{2}}(n-i)\equiv 0\pmod*{p_{1}}.
$$
\end{enumerate}
\end{thm}
\begin{proof}
From the general theory, it is easy to see that the generating functions of the sequences $(B_{p_{1},p_{2}}(n))_{n\in\N}, (C_{p_{1},p_{2}}(n))_{n\in\N}$ are
\begin{align*}
\mathcal{B}_{p_{1},p_{2}}(q)&=\sum_{n=0}^{\infty}B_{p_{1},p_{2}}(n)q^{n}=\prod_{n=1}^{\infty}\left(\frac{1+q^{n^{d}}}{1+q^{(p_{2}n)^{d}}}\right)^{p_{1}^{d}-1},\\
\mathcal{C}_{p_{1},p_{2}}(q)&=\sum_{n=0}^{\infty}C_{p_{1},p_{2}}(n)q^{n}=\prod_{n=1}^{\infty}\left(\frac{1-q^{(p_{2}n)^{d}}}{1-q^{n^{d}}}\right)^{p_{1}^{d}-1}.
\end{align*}

To get the proof of our theorem we recall a well known property of formal power series with integer coefficients: if $f\in\Z[[q]]$ and $p$ is a prime number, then for each $k\in\N_{+}$ we have $f(q^{p^{k}})\equiv f(q)^{p^{k}}\pmod*{p}$.

Let $p_{1}$ be prime. Using the mentioned property we note the following chain of modulo $p_{1}$ equivalences:
\begin{align*}
\sum_{n=0}^{\infty}A_{p_{1},p_{2}}(n)q^{n}&=\prod_{n=1}^{\infty}\frac{(1-q^{(p_{1}n)^{d}})(1-q^{(p_{2}n)^{d}})}{(1-q^{n^{d}})(1-q^{(p_{1}p_{2}n)^{d}})}\equiv \prod_{n=1}^{\infty}\frac{(1-q^{n^{d}})^{p_{1}^{d}}(1-q^{(p_{2}n)^{d}})}{(1-q^{n^{d}})(1-q^{(p_{2}n)^{d}})^{p_{1}^{d}}}\\
              &\equiv \prod_{n=1}^{\infty}\left(\frac{1-q^{n^{d}}}{1-q^{(p_{2}n)^{d}}}\right)^{p_{1}^{d}-1}\pmod*{p_{1}}.
\end{align*}
However, if $p_{1}=2$, then we clearly have
$$
\prod_{n=1}^{\infty}\left(\frac{1-q^{n^{d}}}{1-q^{(p_{2}n)^{d}}}\right)^{2^{d}-1}\equiv \prod_{n=1}^{\infty}\left(\frac{1+q^{n^{d}}}{1+q^{(p_{2}n)^{d}}}\right)^{2^{d}-1}=\sum_{n=0}^{\infty}B_{2,p_{2}}(n)q^{n}
$$
and by comparison of coefficients on both sides of extreme expressions our first result follows.

If $p_{1}\geq 3$, then
$$
\prod_{n=1}^{\infty}\left(\frac{1-q^{n^{d}}}{1-q^{(p_{2}n)^{d}}}\right)^{p_{1}^{d}-1}=\prod_{n=1}^{\infty}\left(\frac{1-q^{(p_{2}n)^{d}}}{1-q^{n^{d}}}\right)^{1-p_{1}^{d}}=\left(\sum_{n=0}^{\infty}C_{p_{1},p_{2}}(n)\right)^{-1},
$$
and thus
$$
\mathcal{A}_{p_{1},p_{2}}(q)\mathcal{C}_{p_{1},p_{2}}(q)\equiv 1\pmod*{p_{1}}.
$$
Consequently, by comparison of coefficients on both sides of the above congruence, we get the second part of our theorem.
\end{proof}

\begin{rem}
{\rm If we assume that $p_{1}, p_{2}$ are both primes then performing the same reasoning as in the proof of the theorem above with respect to modulus $p_{2}$ instead of $p_{1}$, then we get an additional congruence
$$
\mathcal{A}_{p_{1},p_{2}}(q)\mathcal{C}_{p_{2},p_{1}}(q)\equiv 1\pmod*{p_{2}}.
$$
We thus have the following congruence
$$
\mathcal{A}_{p_{1},p_{2}}(q)\equiv \frac{a}{C_{p_{1},p_{2}}(q)}+\frac{b}{C_{p_{2},p_{1}}(q)}\pmod*{p_{1}p_{2}},
$$
where $a, b$ are the unique solutions of the system of congruences
$$
a\equiv 1\pmod*{p_{1}},\quad a\equiv 0\pmod*{p_{2}},\quad b\equiv 0\pmod*{p_{1}},\quad b\equiv 1\pmod*{p_{1}}.
$$
Consequently, we get the congruence
$$
\sum_{i_{1}+i_{2}+i_{3}=n}A_{p_{1},p_{2}}(i_{1})C_{p_{1},p_{2}}(i_{2})C_{p_{2},p_{1}}(i_{3})\equiv aC_{p_{2},p_{1}}(n)+bC_{p_{1},p_{2}}(n)\pmod*{p_{1}p_{2}}.
$$

}
\end{rem}

\begin{rem}
{\rm It should be noted that the number $A_{2,p_{2}}(n)$ has also a different interpretation. More precisely, if $D_{p_{2}}(n)$ denotes the number of partitions into $d$-th powers of odd integers in which no part appears more then $p_{2}^{d}-1$ times, then $A_{2,p_{2}}(n)=D_{p_{2}}(n)$. Indeed, this can be deduced from the general theorem concerning partition ideals \cite[Theorem 8.4]{And0} or can be directly proved by performing simple manipulations of infinite products. We owe this accurate remark to G. Andrews \cite{And2}.
}
\end{rem}

\section{Questions and conjectures concerning the sequence $(p_{d}(n))_{n\in\N}$}\label{sec3}

Let $d\in\N_{\geq 2}$ be fixed. In this section we state some questions and conjectures concerning certain aspects of arithmetic behavior of functions counting $d$-th power partitions.

Let us write
$$
P_{d}(q)=\prod_{n=1}^{\infty}\frac{1}{1-q^{n^{d}}}=\sum_{n=0}^{\infty}p_{d}(n)q^{n}.
$$

Using standard method of logarithmic differentiation we get that
$$
q\frac{P_{d}'(q)}{P_{d}(q)}=\sum_{n=1}^{\infty}\frac{n^{d}q^{n^{d}}}{1-q^{n^{d}}}=\sum_{n=1}^{\infty}\sigma^{(d)}(n)q^{n},
$$
where
$$
\sigma^{(d)}(n)=\sum_{k^{d}|n}k^{d},
$$
with the usual convention that $\sigma^{(d)}(0)=0$. Consequently, after simple manipulations we get the recurrence relation satisfied by the sequence $(p_{d}(n))_{n\in\N}$:
$$
np_{d}(n)=\sum_{i=0}^{n-1}\sigma^{(d)}(i)p_{d}(n-i), \quad p_{d}(0)=1.
$$
This formula can be used to compute $p_{d}(n)$ in terms of $p_{d}(i), i<n$. However, even for relative small values of $n$ the computations are slow. It would be interesting whether there exist a different recurrence formula for $p_{d}(n)$ allowing faster computation for big values of $n$.

For $d=2, 3, 4, 5$ we computed the coefficients $p_{d}(n)$ for $n\leq 10^{5}$. In order to compute these coefficients we use the following approach. First of all we note that
$$
P_{d}(q)-\prod_{i=1}^{\left \lceil 10^{5/d}\right\rceil }\frac{1}{1-q^{i^{d}}}=O(q^{10^{5}+1}),
$$
i.e., instead of working with infinite product $P_{d}(q)$, it is enough to work with a rational function. Thus, if we write
$$
P_{d,k}(q)=\prod_{i=1}^{k}\frac{1}{1-q^{i^{d}}}=\sum_{n=0}^{\infty}p_{d,k}(n)q^{n},
$$
then $p_{d,k}(n)=p_{d}(n)$ for $n\leq k^{d}$. Note that for fixed $k, d$ the sequence $(p_{d,k}(n))_{n\in\N}$ satisfies linear type sequence. More precisely, we have $p_{d,1}(n)=1$ and for $k\geq 2$ the following relations hold:
$$
p_{d,k}(n)=p_{d,k-1}(n),\;\mbox{for}\;n\leq k^{d},\quad p_{d,k}(n)=p_{d,k-1}(n)+p_{d,k}(n-k^d)\;\mbox{for}\;n\geq k^{d}.
$$

We used the above observation to compute $p_{d}(n)$ for $d=2$ and $n\leq 10^5$ and $p_{d}(n)$ for $d=3, 4, 5$ and $n\leq 10^{6}$. Thus, in order to compute $p_{2}(n)$ for $n\leq 10^{5}$ we need to take $k=\lceil 10^{5/2}\rceil=317$. Similarly, in order to compute $p_{d}(n)$ for $d=3, 4, 5$ for $n\leq 10^6$ we take $k=10^2, 32, 16$, respectively.

All computations were performed on a typical laptop with 16 GB of memory and i7 type processor. 

Based on our data we formulate several question and conjectures. We start with the following natural

\begin{conj}\label{conjP0}
Let $d\in\N_{\geq 2}$ and $m\in\N_{\geq 2}$ be given, and take $r\in\{0,\ldots,m-1\}$. Then there are infinitely many values of $n\in\N$ such that $p_{d}(n)\equiv r\pmod*{m}$.
\end{conj}

 The next question concerns asymptotic behavior of the number of solutions of the congruence $p_{d}(n)\equiv r\pmod*{m}$, where $m\in\N_{\geq 2}$ and $r\in\{0,\ldots,m-1\}$.

\begin{ques}\label{quesP1}
Let $d\in\N_{\geq 2}$ and $m\in\N_{\geq 2}$ be given, and take $r\in\{0,\ldots,m-1\}$. Do the values of $p_{d}(n)\bmod{m}$ are equidistributed modulo $m$? More precisely, does the equality
\begin{equation*}
\limsup_{N\rightarrow  +\infty}\frac{\#\{n\leq N:\;p_{d}(n)\equiv r\pmod*{m}\}}{N}=\frac{1}{m}.
\end{equation*}
hold?
\end{ques}

This is very difficult question. We do not even know any equidistribution modulo $m$ result for the classical partition function $p(n)=p_{1}(n)$ for any $m$. In fact, the expectation is that for $m$ co-prime to 6 the values of $p_{1}(n)\bmod{m}$ are not equidistributed. However, it is not clear what to expect in our situation because there is not connections to modular forms and Galois representations as in case of classical partition function.
However, our computations of the quantities
$$
\#\{n\leq m_{d}:\;p_{d}(n)\equiv r\pmod*{m}\}
$$
for $d=2, 3, 4, 5$ seems to confirm Conjecture \ref{conjP0} and the equality stated in Question \ref{quesP1} (at least for $m\leq 10$). See the Table 1, Table 2, Table 3 and Table 4, below.

\begin{equation*}
\begin{array}{c||cccccccccc}
\hline
m\backslash r & 0     & 1     & 2        &3        &4        &5        &6        &        7& 8       & 9   \\
\hline
\hline
  2   & 50299 & 49702 & \text{} & \text{} & \text{} & \text{} & \text{} & \text{} & \text{} & \text{} \\
  3   & 33373 & 33249 & 33379 & \text{} & \text{} & \text{} & \text{} & \text{} & \text{} & \text{} \\
  4   & 25252 & 24695 & 25047 & 25007 & \text{} & \text{} & \text{} & \text{} & \text{} & \text{} \\
  5   & 19940 & 20125 & 19971 & 19955 & 20010 & \text{} & \text{} & \text{} & \text{} & \text{} \\
  6   & 16769 & 16454 & 16735 & 16604 & 16795 & 16644 & \text{} & \text{} & \text{} & \text{} \\
  7   & 14121 & 14272 & 14320 & 14401 & 14257 & 14301 & 14329 & \text{} & \text{} & \text{} \\
  8   & 12679 & 12288 & 12496 & 12371 & 12573 & 12407 & 12551 & 12636 & \text{} & \text{} \\
  9   & 11158 & 11081 & 11033 & 10941 & 11186 & 11239 & 11274 & 10982 & 11107 & \text{} \\
  10  & 10001 & 10025 & 10024 & 9866 & 10085 & 9939 & 10100 & 9947 & 10089 & 9925 \\
\hline
\end{array}
\end{equation*}
\begin{center}
Table 1. Values of $\#\{n\leq 10^{5}:\;p_{2}(n)\equiv r\pmod*{m}\}$ for $0\leq r\leq m-1\leq 9$.
\end{center}

\begin{equation*}
\begin{array}{c||cccccccccc}
\hline
m\backslash r& 0     & 1     & 2        &3        &4        &5        &6        &        7& 8       & 9   \\
\hline
\hline
  2   &   500013 & 499988 & \text{} & \text{} & \text{} & \text{} & \text{} & \text{} & \text{} & \text{} \\
  3   &   333942 & 333563 & 332496 & \text{} & \text{} & \text{} & \text{} & \text{} & \text{} & \text{} \\
  4   &   250099 & 249905 & 249914 & 250083 & \text{} & \text{} & \text{} & \text{} & \text{} & \text{} \\
  5   &   199907 & 200126 & 200490 & 199879 & 199599 & \text{} & \text{} & \text{} & \text{} & \text{} \\
  6   &   167109 & 166685 & 166026 & 166833 & 166878 & 166470 & \text{} & \text{} & \text{} & \text{} \\
  7   &   142501 & 142721 & 142969 & 143340 & 142937 & 142913 & 142620 & \text{} & \text{} & \text{} \\
  8   &   125203 & 124636 & 125023 & 125198 & 124896 & 125269 & 124891 & 124885 & \text{} & \text{} \\
  9   &   111451 & 111275 & 111186 & 111459 & 110992 & 110438 & 111032 & 111296 & 110872 & \text{} \\
  10  &   100033 & 100134 & 100021 & 99625 & 99713 & 99874 & 99992 & 100469 & 100254 & 99886 \\
\hline
\end{array}
\end{equation*}
\begin{center}
Table 2. Values of $\#\{n\leq 10^{6}:\;p_{3}(n)\equiv r\pmod*{m}\}$ for $0\leq r\leq m-1\leq 9$.
\end{center}

\begin{equation*}
\begin{array}{c||cccccccccc}
\hline
m\backslash r& 0     & 1     & 2        &3        &4        &5        &6        &        7& 8       & 9   \\
\hline
\hline
  2   &   500517 & 499484 & \text{} & \text{} & \text{} & \text{} & \text{} & \text{} & \text{} & \text{} \\
  3   &   333153 & 333474 & 333374 & \text{} & \text{} & \text{} & \text{} & \text{} & \text{} & \text{} \\
  4   &   250463 & 249010 & 250054 & 250474 & \text{} & \text{} & \text{} & \text{} & \text{} & \text{} \\
  5   &   200555 & 199837 & 199524 & 200091 & 199994 & \text{} & \text{} & \text{} & \text{} & \text{} \\
  6   &   166388 & 166699 & 167354 & 166765 & 166775 & 166020 & \text{} & \text{} & \text{} & \text{} \\
  7   &   143174 & 142713 & 143172 & 142658 & 142908 & 142621 & 142755 & \text{} & \text{} & \text{} \\
  8   &   125224 & 124544 & 125595 & 125373 & 125239 & 124465 & 124459 & 125101 & \text{} & \text{} \\
  9   &   111012 & 111100 & 111214 & 111263 & 111238 & 111071 & 110878 & 111136 & 111089 & \text{} \\
  10  &   100310 & 99810 & 99660 & 99706 & 100135 & 100245 & 100027 & 99864 & 100385 & 99859 \\
\hline
\end{array}
\end{equation*}
\begin{center}
Table 3. Values of $\#\{n\leq 10^{6}:\;p_{4}(n)\equiv r\pmod*{m}\}$ for $0\leq r\leq m-1\leq 9$.
\end{center}

\begin{equation*}
\begin{array}{c||cccccccccc}
\hline
m\backslash r& 0     & 1     & 2        &3        &4        &5        &6        &        7& 8       & 9   \\
\hline
\hline
  2   &    500386 & 499615 & \text{} & \text{} & \text{} & \text{} & \text{} & \text{} & \text{} & \text{} \\
  3   &    334253 & 332498 & 333250 & \text{} & \text{} & \text{} & \text{} & \text{} & \text{} & \text{} \\
  4   &    249768 & 249985 & 250618 & 249630 & \text{} & \text{} & \text{} & \text{} & \text{} & \text{} \\
  5   &    199971 & 199526 & 200089 & 200380 & 200035 & \text{} & \text{} & \text{} & \text{} & \text{} \\
  6   &    167002 & 166054 & 166940 & 167251 & 166444 & 166310 & \text{} & \text{} & \text{} & \text{} \\
  7   &    143141 & 142701 & 142907 & 143029 & 142768 & 143046 & 142409 & \text{} & \text{} & \text{} \\
  8   &    124187 & 125010 & 125168 & 125302 & 125581 & 124975 & 125450 & 124328 & \text{} & \text{} \\
  9   &    111905 & 111078 & 110740 & 110779 & 111233 & 111095 & 111569 & 110187 & 111415 & \text{} \\
  10  &    100264 & 99955 & 100250 & 100380 & 100301 & 99707 & 99571 & 99839 & 100000 & 99734 \\
\hline
\end{array}
\end{equation*}
\begin{center}
Table 4. Values of $\#\{n\leq 10^{6}:\;p_{5}(n)\equiv r\pmod*{m}\}$ for $0\leq r\leq m-1\leq 9$.
\end{center}

In the context of classical partition function of Euler, i.e., $p(n)=p_{1}(n)$, where there are plenty of triples $a, b, m$, where $m\in\N_{\geq 5}$ and $a, b\in\N_{+}$, such that $p(an+b)\equiv 0\pmod*{m}$ for all $n\in\N$, one can ask the following:

\begin{ques}
Let $d\in\N_{\geq 2}$ be fixed. Does there exist $m\in\N_{\geq 2}, r\in\{0,\ldots, m-1\}$, and positive integers $a, b$ such that for each $n\in\N$ we have $p_{d}(an+b)\equiv r\pmod*{m}$?
\end{ques}

In the considered range we were unable to find a single quadruple $(m, r, a, b)$ and $d\in\{2,3,4,5\}$ such that $p_{d}(an+b)\equiv r\pmod*{m}$ for $n=0,1\ldots, 100$. In order to guarantee that $100a+b\leq 10^{5}$ we considered the range $a\in\{2,\ldots, 999\}, b\in\{0,\ldots, a-1\}$.
This may suggest that even if there are quadruplets $(m, r, a, b)$ such that $p_{d}(an+b)\equiv r\pmod*{m}$ for all $n$, they are rare.

\bigskip

Let us recall that a sequence $(a_{n})_{n\in\N}$ is convex, if $2a_{n}\leq a_{n-1}+a_{n+1}$ for $n\geq 1$. We formulate the following general

\begin{conj}
Let $d\in\N_{\geq 2}$. Then there is an integer $N_{d}$ such that for all $n\geq N_{d}$ we have $2p_{d}(n)\leq (p_{d}(n-1)+p_{d}(n+1))\left(1-\frac{1}{n^d}\right)$. In particular the sequence $(p_{d}(n))_{n\geq N_{d}}$ is convex.
\end{conj}

The above conjecture can be seen as a natural generalization of log-concavity of the classical partition function $p(n)=p_{1}(n)$.  We checked that
$$
2p_{2}(n)\leq (p_{2}(n-1)+p_{2}(n+1))\left(1-\frac{1}{n^2}\right)\quad\mbox{for} \quad n\in\{379, \ldots, 10^{5}-1\},
$$
and
$$
2p_{3}(n)\leq (p_{3}(n-1)+p_{3}(n+1))\left(1-\frac{1}{n^3}\right)\quad\mbox{for} \quad n\in\{6769, \ldots, 10^{6}-1\},
$$
and
$$
2p_{4}(n)\leq (p_{4}(n-1)+p_{4}(n+1))\left(1-\frac{1}{n^4}\right)\quad\mbox{for} \quad n\in\{239603, \ldots, 10^{6}-1\},
$$
i.e., we believe that $N_{2}=379, N_{3}=6769, N_{4}=239603$.

It seems that the number $N_{5}$ (if is exists) is $\geq 10^{6}$.

\bigskip

Let us recall that a sequence  $(a_{n})_{n\in\N}$ of positive reals is log-concave, if $a_{n}^{2}\geq a_{n-1}a_{n+1}$ for $n\geq 1$, i.e., the sequence $(-\log a_{n})_{n\in\N}$ is convex. We formulate the following general.

\begin{conj}
Let $d\in\N_{\geq 2}$. Then there is an integer $M_{d}$ such that for all $n\geq M_{d}$ we have
$$
p_{d}^{2}(n)\geq p_{d}(n-1)p_{d}(n+1)\left(1+\frac{1}{n^{d}}\right).
$$
In particular the sequence $(p_{d}(n))_{n\geq M_{d}}$ is log-concave.
\end{conj}

We checked that
$$
p_{2}^{2}(n)\geq p_{2}(n-1)p_{2}(n+1)\left(1+\frac{1}{n^{2}}\right)\quad\mbox{for} \quad n\in\{1086, \ldots, 10^{5}-1\},
$$
and
$$
p_{3}^{2}(n)\geq p_{3}(n-1)p_{3}(n+1)\left(1+\frac{1}{n^{3}}\right)\quad\mbox{for} \quad n\in\{15656, \ldots, 10^{6}-1\},
$$
and
$$
p_{4}^{2}(n)\geq p_{4}(n-1)p_{4}(n+1)\left(1+\frac{1}{n^{4}}\right)\quad\mbox{for} \quad n\in\{637855, \ldots, 10^{6}-1\},
$$
i.e., we believe that $M_{2}=1042, M_{3}=15656, M_{4}=637855$.

It seems that the number $ M_{5}$ (if it exists) is $\geq 10^{6}$. It is very likely that using the classical asymptotic formula for $p_{d}(n)$ of Wright \cite{MW} or its current improvements, the above conjecture can be resolved. It should be noted that an analogous result for Euler partition function $p(n)$, i.e., the case of $d=1$ of the above conjecture, was proved by DeSalvo and Pak \cite{DeSP} and recently generalized by Hou and Zhang \cite{HZ}.

\bigskip

\noindent  Maciej Ulas, Jagiellonian University, Faculty of Mathematics and Computer Science, Institute of Mathematics, {\L}ojasiewicza 6, 30 - 348 Krak\'{o}w, Poland\\
e-mail:\;{\tt  maciej.ulas@uj.edu.pl}
\bigskip

 \end{document}